\documentclass[11pt]{amsart}
\usepackage{amsmath}
\usepackage{amsfonts}

\usepackage{amscd}
\usepackage{amsthm}
\usepackage{amssymb} \usepackage{latexsym}
\usepackage{amsmath,mathrsfs}
\usepackage{euscript}
\usepackage{epsfig}
\usepackage{graphics}
\usepackage{array}
\usepackage{enumerate}
\usepackage{dsfont}
\usepackage{color}
\usepackage{wasysym}

\usepackage{pdfsync}
\usepackage[colorlinks,citecolor=red,pagebackref,hypertexnames=false, breaklinks]{hyperref}

\calclayout
\allowdisplaybreaks

\newcommand{\bel}[1]{\begin{equation}\label{#1}}

\newcommand{\be}{\begin{equation}}

\newcommand{\ba}{\begin{eqnarray}}
\newcommand{\ea}{\end{eqnarray}}

\newcommand{\qe}{\end{equation}}

\newcommand{\N}{{\mathbb N}}

\newcommand{\dd}{\mathrm{d}}

\newcommand{\Hmm}[1]{\leavevmode{\marginpar{\tiny%
$\hbox to 0mm{\hspace*{-0.5mm}$\leftarrow$\hss}%
\vcenter{\vrule depth 0.1mm height 0.1mm width \the\marginparwidth}%
\hbox to
0mm{\hss$\rightarrow$\hspace*{-0.5mm}}$\\\relax\raggedright #1}}}

\theoremstyle{plain}
\newtheorem{thm}{Theorem}[section]
\newtheorem{lemma}[thm]{Lemma}

\newtheorem{prop}[thm]{Proposition}

\theoremstyle{definition}
\newtheorem{defn}[thm]{Definition}
\newtheorem{rem}[thm]{Remark}

\numberwithin{equation}{section}

\begin{document}

\date{\today}
\title[Liouville theorems for ancient solutions on graphs]{Liouville theorems for ancient solutions of subexponential growth to the heat equation on graphs}
\author{Bobo Hua}
\email{bobohua@fudan.edu.cn}
\address{Bobo Hua: School of Mathematical Sciences, LMNS,
Fudan University, Shanghai 200433, China; Shanghai Center for
Mathematical Sciences, Fudan University, Shanghai 200438,
China.}

\author{Wenhao Yang}
\email{whyang23@m.fudan.edu.cn}
\address{Wenhao Yang: School of Mathematical Sciences, Fudan University, Shanghai 200433, China.}


\begin{abstract}
    Mosconi \cite{mosconi2021} proved Liouville theorems for ancient solutions of subexponential growth to the heat equation on a manifold with Ricci curvature bounded below.
    We extend these results to graphs with bounded geometry: for a graph with bounded geometry, any nonnegative ancient solution of subexponential growth in space and time to the heat equation is stationary,
    and thus is a harmonic solution.


\end{abstract}

\maketitle


\section{Introduction}


In 1975, Yau \cite{Yau75} proved the celebrated Liouville theorem  for harmonic functions: on a Riemannian manifold with nonnegative Ricci curvature any positive harmonic function is constant.
Later, Cheng and Yau \cite{ChengYau75} obtained the gradient estimate for positive harmonic functions, which implies that any harmonic function of sublinear growth on {such a manifold} is constant.
In 1997, Colding-Minicozzi \cite{ColdingMAnnals97} proved Yau's conjecture that the space of polynomial growth harmonic functions is finite dimensional.
Liouville theorems for harmonic functions on manifolds have received much attention in the literature, 
e.g. \cite{Sullivan83,Anderson83,LiTam87,Lyons87,Grigoryan90,Benjamini91,Grigoryan91,Saloff92,WangFY02,Erschler04,Li12,Brighton13}.

In 1986, Li and Yau \cite{LiYau86} proved the gradient estimate for positive solutions to the heat equation,  nowadays called the Li-Yau inequality.
An ancient solution to the heat equation is a solution defined on the time interval $(-\infty, t_0)$ for some $t_0 \in \mathbb{R}$.
As a corollary of the Li-Yau inequality, any bounded ancient solution on a manifold with nonnegative Ricci curvature is constant.
Souplet and Zhang \cite{souplet2006sharp} proved interesting Liouville theorems for positive eternal solutions (i.e. solutions defined in all space and time) of subexponential growth and also for general eternal solutions of sublinear growth on a manifold with nonnegative Ricci curvature. For such a manifold, Lin and Zhang \cite{lin2019ancient}  further proved that an ancient solution of polynomial growth to the heat equation is a polynomial in time.
The latter was extended to manifolds with polynomial volume growth by Colding-Minicozzi \cite{colding2021optimal}.

In 2020, Mosconi proved much more general Liouville theorems for nonnegative ancient solutions of subexponential growth to the heat equation on a manifold with Ricci curvature bounded below, which extended the result in \cite{souplet2006sharp}.

\begin{thm}[\cite{mosconi2021}]\label{thm:mos}
    Suppose that $M$ is a complete noncompact Riemannian manifold with $Ric_M \geq -K (K \geq 0)$.
    Let $u \in C^{\infty}\left(M \times\left(-\infty, t_0\right)\right)$ 
    be a nonnegative ancient solution to the heat equation on $M$.
    Let $p \in M$ and $\rho$ be the distance function. Then the following hold:

    (1) For $K = 0$, if there exists $t \in\left(-\infty, t_0\right)$ such that $u(x, t) = e^{o(\rho(x, p))}
    $ as $\rho(x, p) \rightarrow+\infty$,
    then $u$ is constant.

    (2) For $K > 0$, if $u(x, t) = e^{o(\rho(x, p)-t)} $ as $ \rho(x, p)\rightarrow+\infty$ and $t \rightarrow-\infty$,
    then $u$ is independent of $t,$ and is a harmonic function.
\end{thm}

We aim to extend the above theorem to graphs. We first introduce the setting of graphs. Assume $G = (V,E)$ is a connected, locally finite, simple and undirected graph, consisting of the set of vertices $V$ and the set of edges $E$.
Vertices $x$ and $y$ are called neighbors, denoted by $x \sim y$, if there is an edge connecting them, i.e. ${\{x,y\} \in E}$.
For any vertex $x$, deg$x$ is the number of neighbors of $x$.

Let
$$
    \begin{aligned}
        w: E       & \rightarrow(0,+\infty),    \\
        e=\{x, y\} & \mapsto w_e=w_{xy}=w_{yx},
    \end{aligned}
$$
be the weight function on edges, and
$$
    \begin{aligned}
        m: V & \rightarrow(0,+\infty), \\
        x    & \mapsto m_x,
    \end{aligned}
$$ be the weight function on vertices. We call $G=(V,E,w,m)$ a weighted graph. For a weighted graph $G=(V,E,w,m)$,
the Laplace operator $\Delta$ on $G$ is defined as follows: for any function
$f : V \to \mathbb{R}, $
$$
    \Delta f(x)=\sum_{\substack{y \sim x \\ y \in V}} \frac{w_{xy}}{m_x}(f(y)-f(x)), \   \text{for all} \ x \in V.
$$

We denote by
$\rho(x, y) =\min \{ l | x = z_0 \sim ... \sim z_l = y\}$
the combinatorial distance between vertices $x$ and $y,$ by $B_n\left(x_0\right):=\left\{x \in V: \rho\left(x, x_0\right) \leq n\right\}$ and
$\partial B_n\left(x_0\right):=\left\{x \in V : \rho\left(x, x_0\right) = n\right\}$ for $n\in \N,$ the ball and sphere centered at $x_0$ respectively.

We recall Liouville theorems for harmonic functions on graphs.
The Liouville theorem for positive harmonic functions is proved
by Delmotte \cite{DelmottePolynomial98} on a
graph satisfying the volume doubling property and the Poincaré inequality, and by Bauer et al.\cite{bauer2015li,HLLY19} on a graph with
nonnegative curvature in the sense of CDE or CDE'. The Liouville theorem for bounded harmonic functions is proved by the first author
\cite{hua2019Liouville} for a graph with nonnegative Bakry-Émery curvature, and by Jost, M\"unch and Rose \cite{jost2019liouville}
for a graph with nonnegative Ollivier curvature.

In this paper, we prove Liouville theorems for nonnegative ancient solutions of subexponential growth to the heat equation on graphs. We first introduce the definition of graphs with bounded geometry.

\begin{defn}\label{BG}
    We call a weighted graph $G=(V,E,w,m)$ has \emph{bounded geometry} if
    there exists $c_0>0$ such that

    $$
        \left\{\begin{aligned}
            c_0^{-1} \leq w_e \leq c_0, \       & \forall e \in E , \\
            c_0^{-1} \leq m_x \leq c_0, \       & \forall x \in V , \\
            \operatorname{deg} x \leq c_0, \  & \forall x \in V.
        \end{aligned}\right.
    $$
\end{defn}

The following is the main result.
\begin{thm}\label{main}
    Let $G=(V,E,w,m)$ be an infinite graph with bounded geometry,
    $u \in C^{\infty}\left(V \times\left(-\infty, t_0\right)\right)$
    be a nonnegative ancient solution to the heat equation on $G$.
    Assume $x_0 \in V $, and let $\lambda_1(G)$ be the bottom
    of the spectrum of $G$. Then the following hold:

    (1) If $\lambda_1(G)=0,$ and there exists $t \in\left(-\infty, t_0\right)$ such that $$u(x, t)=e^{o(\rho(x, x_0))},\rho(x, x_0) \rightarrow+\infty,$$
    then $u$ is independent of $t,$ and is a harmonic function on $G$.

    (2) If $\lambda_1(G)>0,$ and there exist  $t \in\left(-\infty, t_0\right)$ and $y_0 \in V$ such that
    $$u(x, t)=e^{o(\rho(x, x_0))},\rho(x, x_0)\rightarrow+\infty,
   \  {\rm and}\ 
    u(y_0, t)=e^{o(-t)}, t \rightarrow-\infty,$$
    then $u$ is independent of $t,$ and is a harmonic function on $G$.
\end{thm}


\begin{rem}
\begin{enumerate}
\item
  The above theorem reduces the Liouville property of ancient solutions of subexponential growth to that of harmonic functions.  As a corollary, any bounded ancient solution is constant for a graph with nonnegative Bakry-\'Emery curvature or Ollivier curvature.
  \item Compared with Theorem~\ref{thm:mos}, we use the bottom of the spectrum $\lambda_1(G),$ instead of the curvature lower bound, to classify the cases.
    \end{enumerate}
\end{rem}

We follow the proof strategy of Mosconi \cite{mosconi2021}:
First, by Choquet Representation Theorem of convex cones, nonnegative ancient solutions to the heat equation can be
written as integral of extreme points of the cone.
Mosconi used the Harnack inequality for positive solutions to the heat equation on manifolds with Ricci curvature bounded below to classify those extreme points.
In particular, any extreme point can be expressed as $e^{\lambda t}w(x)$ in the form of separation of variables, where $w$ is a generalized eigenfunction satisfying $ \Delta w = \lambda w$.
Moreover, one can show that these generalized eigenfunctions $w(x)$ grow at least exponentially.
Thus, it contradicts the assumption of subexponential growth of ancient solutions.
The key point is to classify these extreme points.
However, the Harnack inequality on graphs with curvature conditions remains still open.
In this paper, we replace the assumption of lower bounded curvature with the assumption of bounded geometry.
In this case, we use the local Harnack inequality for the heat equation by Lin et al. \cite{lin2017gradient}, which suffices to classify extreme points.
Moreover, by the assumption of bounded geometry we get the exponential growth rate of generalized eigenfunctions
on such graphs, and hence prove the theorem.

\section{The Representation of Nonnegative Ancient Solutions to the Heat Equation}

First, we introduce some notations.
Let $X$ be a real linear space and let $ A \subset X$.
$A$ is called a convex cone if it is a convex set satisfying
    $\lambda x \in A,$ for all $x \in A, \lambda > 0;$
$A$ is called a proper convex cone if it is a convex cone satisfying
    $A \cap -A = \{0\} \text{ or } \varnothing.$
A topological linear space is called a locally convex space provided that there exists a local base at the origin consisting of convex sets.

Similar to extreme points of convex sets, we can define extreme points and extreme rays of convex cones. Let $A$ be a convex cone. For any $a \in A$, denote by $r_a = \{\lambda a: \lambda > 0\}$ the ray passing through $a$. For $a \in A$, we call $r_a$ an extreme ray provided that $u \in r_a$, $v \in A$ and $u - v \in A$ imply $v = ku$ for some $k > 0.$ Denote by $\operatorname{Ext}(A)$ the set of extreme points of $A,$ which is the union of all extreme rays.

Let $C:= \left\{u \in C^{\infty}\left(V \times \left(-\infty, t_0\right)\right): \Delta u = \partial_t u, u \geq 0\right\}$
be the set of nonnegative ancient solutions to the heat equation, which is a proper convex cone.

The following is the local Harnack inequality for the heat equation
on a graph with bounded geometry.

\begin{lemma}[Local Harnack inequality, Theorem 3 in \cite{lin2017gradient}]\label{Harnack}
  Let $G = (V, E, w, m)$ be a graph with bounded geometry.
    If $u \in C^{\infty}\left(V \times \left(-\infty, +\infty\right)\right)$
    satisfies $\Delta u = \partial_t u$ and $u \geq 0$,
    then for all $x, y \in V$ and $t_1 < t_2$,
    $$
        u\left(x, t_1\right) \leq  u\left(y, t_2\right) \exp \left\{
        C_1\left(t_2 - t_1\right)
        +\frac{C_2(\rho^2(x, y))}{t_2 - t_1} \right\},
    $$
    where $C_1$ and $C_2$ are constants.
\end{lemma}

By the local Harnack inequality, one can classify 
extreme points 
of $C$.
A function $w$ is called an (generalized) eigenfunction, or $\lambda$-eigenfunction if $\Delta w=\lambda w.$
\begin{prop}\label{Ext}
    If $0 \neq u \in \operatorname{Ext}(C)$,
    then there exists a unique representation $u(x, t) = e^{\lambda t} w(x)$,
    where $\lambda \in \mathbb{R}$ and $w$ is a $\lambda -$eigenfunction.
\end{prop}


\begin{proof}
    Note that
    \begin{equation}\label{minimal}
        \operatorname{Ext}(C)=\{u \in C: \text{if } \exists \  v \in C , v \leq  u , v=k u \ (k>0)\}.
    \end{equation}


  For $u\neq 0,$ by Lemma~\ref{Harnack}, one can show that $u>0.$   Moreover, letting $x=y, t_2=t$, and $t_1=t-a \ (a>0)$ in Lemma \ref{Harnack}, we have
    $$u(x, t-a) \leq u(x, t) \cdot e^{C_1 a}.$$

   Set $v(x, t)=u(x, t-a)$.
    Clearly, $v \in C$, and it satisfies
    $$v(x, t) \leq e^{C_1 a} u(x, t).$$ By (\ref{minimal}), we have
    $v(x, t)=k(u, a) u(x, t)$, where the constant $k$ depends only on
    the function $u$ and $a$. Hence,
    $$    \partial_t u(x, t)=\lim_{a\to 0} \frac{u(x, t)-u(x, t-a)}{t-(t-a)} =\lim_{a\to 0}\frac{1-k(u, a)}{a} u(x, t)=:\lambda u(x,t),
    $$ where $\lambda$ depends only on $u.$

By integrating over time, we have
    $$u(x, t) =e^{\lambda t} w(x).$$ By the heat equation of $u,$
    \begin{equation*}
        \Delta w=\lambda w.
    \end{equation*}

This yields the representation, and the uniqueness follows directly.

\end{proof}

To obtain the representation theorem, we refer to Choquet Representation Theorem for convex cones.

\begin{lemma}[Choquet Representation Theorem, Theorem 30.22 in \cite{choquet1969lectures}]\label{Choquet} Let $X$ be a Hausdorff locally convex space, $A \subset X$ be a weakly complete, metrizable, and proper convex cone. For any $a \in A$, there exists a Borel probability measure $\mu$ with $\operatorname{supp} \mu \subseteq \operatorname{Ext}(A)$ such that
    $$f(a) = \int_{y \in \operatorname{Ext}(A)} f(y) \dd\mu(y), \quad \text{for all } f \in X^*.$$ 
\end{lemma}

\begin{prop}\label{Cond}
    $C$ is a proper convex cone on $\mathbb{R}^{V \times(-\infty, t_0)}$
    under the topology of pointwise convergence, and
    satisfies all the conditions of Lemma \ref{Choquet}.
\end{prop}

\begin{proof}

    (1) $\mathbb{R}^{V \times(-\infty, t_0)}$ is Hausdorff.

    The uniqueness of limits of sequences in this topology yields the result.

    (2) $\mathbb{R}^{V \times(-\infty, t_0)}$ is locally convex.

    For any $a \in V \times(-\infty, t_0)$ and $u \in X$, let $p_a(u) = u(a)$.
    One easily checks that $p_a(u)$ is a semi-norm on
    $\mathbb{R}^{V \times(-\infty, t_0)}$. Let $\mathscr{P}
        = \{p_a : a \in V \times(-\infty, t_0)\}$ be a family of semi-norms on $\mathbb{R}^{V \times(-\infty, t_0)}$. One can show that $\mathscr{P}$ is a separating family,
    and by \cite[Theorem 1.37]{rudin1973functional}, $\mathscr{P}$
    generates a balanced and convex local base at the origin.

    (3) $C$ is weakly complete.

   Let $\{u_n\}$ be a weak Cauchy sequence in $C$, which means that $\{Tu_n\}$ is a Cauchy sequence for all $T \in C^*$. Fix $(x, t) \in V \times(-\infty, t_0)$, and define $T_{(x, t)}u = u(x, t),\forall u\in C$. Thus, $\{u_n(x, t)\}_{n=1}^{+\infty}$ is a Cauchy sequence. We define
    $$u_\infty(x,t):=\lim_{n\to \infty}u_n(x, t),\quad \forall (x, t) \in V \times(-\infty, t_0).$$
     It suffices to show that $u_0 \in C$. 


By the local Harnack inequality, Lemma \ref{Harnack},
for any $t\in(t_1,t_2)$ with $0<t_1<t_2<\infty$ and $k=0,1$ or $2,$ 
$$|\partial_t^k u_n(x,t)|=|\Delta^k u_n(x,t)|\leq C\max_{B_k(x)}u_n(\cdot, t)\leq 
 Cu_n(x,t_2)\leq \tilde{C}<\infty.$$
By passing to a subsequence, we get $C^1$ convergence in time for $u_n,$ and hence $u_\infty$ is a solution to the heat equation. It is obvious that $u_\infty\in C.$ 

    (4) $C$ is metrizable.

    Take a countable dense subset $D \subset V \times(-\infty, t_0).$ We can equip the elements of $C$ with a metric only dependent on the value of functions taken on $D$ such that it preserves the pointwise convergence on $D$.
  By the local Harnack inequality, Lemma \ref{Harnack}, in $C$,
   the pointwise convergence on $D$ coincides with the pointwise convergence on $V \times(-\infty, t_0)$, and thus we can define a metric on $C$ compatible with its original topology.
\end{proof}

\begin{prop}\label{Representation}
    For $u \in C$, there exist a Borel probability measure $\nu$ on $\mathbb{R}$,
    and  a family of nonnegative functions $\{w_\lambda\}$(dependent on $u$)
    satisfying $\Delta w_\lambda = \lambda w_\lambda$, such that
    \begin{equation}\label{27}
        u(x, t) = \int_{\mathbb{R}} e^{\lambda t} w_\lambda(x) \dd\nu(\lambda).
    \end{equation}

    Furthermore, the function $\lambda \mapsto w_\lambda(x)$ is a Borel function.
\end{prop}

\begin{proof}
By Lemma \ref{Choquet}, for any $u \in C$, there exists a Borel probability measure $\mu$, with $\operatorname{supp}\mu \subseteq \operatorname{Ext}(C)$ such that for any continuous linear functional $f$,
    $$f(u)=\int_{v \in \operatorname{Ext}(C)} f(v)  \dd \mu(v).$$

    Fix $(x, t) \in V \times\left(-\infty, t_0\right)$ and define $f(h) = h(x, t), \forall h\in C. $ Then we have
    $$u(x,t) = \int_{v \in \operatorname{Ext}(C)} v(x,t) \dd\mu(v).$$

    By Proposition \ref{Ext},

    $$u(x,t) = \int_{v \in \operatorname{Ext}(C)} e^{\lambda t} w(x) \dd\mu(v).$$

    Next, we want to rewrite the integral over elements of
    $\operatorname{Ext}(C)$ as an integral over $\lambda$.
    The idea is to classify the elements in $\operatorname{Ext}(C)$ based on $\lambda$ and construct a new measure and integral.

    We introduce the classification function
    $$
        \begin{aligned}
            \varphi: \operatorname{Ext}(C)     & \rightarrow \mathbb{R}, \\
            0 \neq v(x, t) =e^{\lambda t} w(x) & \mapsto \lambda .
        \end{aligned}
    $$ In particular, we define $\varphi(0)=0,$ and by Proposition \ref{Ext},
    the function $\varphi$ is well-defined.

    Now we prove that $C$ is a Polish space.
    Since $C$ is metrizable, its weak completeness implies the completeness. The metric mentioned in Proposition~\ref{Cond}(4) naturally induces the C2 property and separability, making $C$ a Polish space.
    Moreover, $\operatorname{Ext}(C)$, being a $G_\delta$ set of $C$, is also a Polish space and therefore a Radon space.

    Next, we will prove that $\varphi$ is continuous except at 0, hence measurable.
Suppose $\{v_n\} \subset \operatorname{Ext}(C)$ converges to $0 \neq  {v_0} \subset \operatorname{Ext}(C)$. By Proposition \ref{Ext}, we can assume that for any $n \in \N_0$, $e^{\lambda_n t} \widetilde{w_n}(x)=v_n(x, t)$, where $\widetilde{w_n}$ satisfies $\Delta \widetilde{w_n}=\lambda_n \widetilde{w_n}$ and $\widetilde{w_0} \neq 0$. For any $x \in V $ and $ t \in (-\infty, t_0)$, we have
    \begin{equation}\label{lim1}
        e^{\lambda_n t} \widetilde{w_n}(x)=v_n(x, t) \rightarrow v_0(x, t)=e^{\lambda_0 t} \widetilde{w_0}(x) , n \rightarrow+\infty.  
    \end{equation}
    Clearly, (\ref{lim1}) still holds for $t-1$. Thus,
    \begin{equation}\label{lim2}
        e^{\lambda_n (t-1)} \widetilde{w_n}(x) \rightarrow  e^{\lambda_0 (t-1)} \widetilde{w_0}(x) , n \rightarrow+\infty.  
    \end{equation}
    Since $\widetilde{w_0} \neq 0$, deviding (\ref{lim1}) by (\ref{lim2}), we have
    $e^{\lambda_n} \rightarrow e^{\lambda_0} ,n \rightarrow+\infty,$
    and thus
    $\lambda_n \rightarrow \lambda_0 ,n \rightarrow+\infty.$ Therefore, $\varphi$ is continuous except at 0, thus Borel measurable.
    
    The measurable function $\varphi$ induces a decomposition \begin{equation}\label{Deco}
        \operatorname{Ext}(C)=\bigcup_{\lambda \in \mathbb{R}}\varphi^{-1}(\lambda).
    \end{equation} By the Disintegration Theorem, there exist conditional probability measures $\left\{\mu_\lambda\right\}$ with $\operatorname{supp} \mu_\lambda \subseteq \varphi^{-1}(\lambda)$, and a Borel probability measure $\nu$ on $\mathbb{R}$ such that

    \begin{equation}\label{bs}
        u=\int_{\mathbb{R}} \int_{v \in\varphi^{-1}(\lambda)} v \dd\mu_\lambda(v) \dd\nu(\lambda).
    \end{equation} Moreover, $\lambda \mapsto \mu_\lambda$ is a Borel function and if we define
    $
        u_\lambda:=\int_{v \in \varphi^{-1}(\lambda)} v \dd\mu_\lambda(v) =\int_{v \in \operatorname{Ext}(C)}  v \dd\mu_\lambda(v),
    $
    then $\lambda \mapsto u_\lambda$ is also a Borel function.

    Let $w_\lambda = e^{-\lambda t}u_\lambda$.
    Since $u_\lambda \mapsto e^{-\lambda t}u_\lambda $ is continuous in the topology of pointwise convergence, $\lambda \mapsto w_\lambda$ is Borel.
    By Proposition \ref{Ext}, for any
    $ v \in \varphi^{-1}(\lambda) \subset \operatorname{Ext}(C), e^{-\lambda t}v$ is independent of $t$.

        Thus, $$w_\lambda:= e^{-\lambda t}u_\lambda=e^{-\lambda t}\int_{v \in \varphi^{-1}(\lambda)} v \dd\mu_\lambda(v)$$
    is independent of $t.$ Moreover, by $\partial_t v=\lambda v,$



    
       $$
\Delta w_\lambda =e^{-\lambda t}\Delta u_\lambda=e^{-\lambda t}\int_{v \in \varphi^{-1}(\lambda)}{\Delta v \dd \mu_\lambda(v)}
=e^{-\lambda t}\int_{v \in \varphi^{-1}(\lambda)} \partial_t v \dd \mu_\lambda(v) 
 =\lambda w_\lambda.
    $$





    This proves the proposition.
\end{proof}

\begin{rem}\label{Rem}
    Furthermore, we can also require $$\operatorname{supp}\nu \subseteq \Lambda := \{ \lambda: \Delta w=\lambda w\ 
    \mathrm{has\ a\ nonzero\ nonnegative\ solution\ on}\  G\}.$$

    In fact, by Lemma~\ref{Spec} and Proposition~\ref{Zer},
    $\Lambda =[-\lambda_1(G),+\infty).$
    Hence, $\Lambda$ is a closed set on $\mathbb{R}$, and is naturally a Radon space.
    By the definition of $\Lambda$,
    we know that $\forall \lambda \notin \Lambda, \ \varphi^{-1}(\lambda)=\varnothing$.
    Hence, we only need to modify the decomposition (\ref{Deco}) to
    $$
        \operatorname{Ext}(C)=\bigcup_{\lambda \in \Lambda}\varphi^{-1}(\lambda),
    $$
    and change $\mathbb{R}$ to $\Lambda$.
\end{rem}

\begin{rem}\label{24}
    We can also require that $\forall \ \lambda \in \operatorname{supp}
        \nu, \ w_\lambda>0$. In fact, by Proposition \ref{Zer} below, $\left\{\lambda: w_\lambda = 0\right\}^c=\left\{\lambda: w_\lambda \neq 0\right\}
        =\left\{\lambda: w_\lambda>0\right\}$
    is measurable since $\lambda \mapsto w_\lambda$ is a Borel function. Thus, we can restrict $\nu$ on this set.
\end{rem}

\section{Estimates of the Growth Rate of the generalized eigenfunction}

The following proposition is elementary.
\begin{prop}\label{Zer}
    On a weighed graph $G=(V,E,w,m)$, let $w$ be a nonnegative $\lambda-$eigenfunction on $G$,
    $\lambda \in \mathbb{R}.$
    If there exists $x_0 \in V$ s.t. $w(x_0)=0$, then $w(x)=0, \forall x \in V.  $
\end{prop}

\begin{proof}

  By the equation $\Delta w=\lambda w$ at the vertex $x_0,$ we get $w(y)=0, \forall y \sim x_0.$ Applying the above result for these $y$ and using the iteration,
 we prove 
    $w(x)=0, \forall x \in V$ by the connectedness of $G.$ 

\end{proof}

We recall a well-known result.
\begin{lemma}[Agmon-Allegretto-Piepenbrink Theorem, see e.g. Theorem 4.14 in \cite{keller2021graphs}]\label{Spec}
    Assume graph $G=(V,E,w,m)$
    is infinite, locally finite and connected.
    Let $\lambda_1(G)$ denote the bottom of the spectrum of Laplacian on $G$
    and $\lambda \in \mathbb{R}$.
    Then the following are equivalent:

    (1) $\lambda \geq -\lambda_1(G)$.

    (2) $\Delta w=\lambda w$ has a positive solution $w$.
\end{lemma}

The maximum principle is well-known.

\begin{lemma}[Lemma 1.39 in \cite{grigor2018introduction}]\label{Maximal}
    Assume $G=(V,E,w,m)$ is an infinite and connected graph,
    and $w$ is a subharmonic function, i.e. $\Delta w \geq 0$.
    Then $\forall x_0 \in V, n \in \mathbb{N}$,
    \begin{equation}\label{Max}
        \max _{\partial B_n\left(x_0\right)} w=\max _{B_n\left(x_0\right)} w.
    \end{equation}
\end{lemma}

For $x_0 \in V, w:V \to  \mathbb{R}$, we denote by
$$
    M_n:=\max \limits _{B_n\left(x_0\right)} w.
$$

We prove the estimates on the growth rate of $M_n$ for a nonnegative generalized eigenfunction.

\begin{prop}\label{Grow0}
    Assume that the graph $G=(V,E,w,m)$ has bounded geometry,
    $x_0 \in V$, $\lambda \geq 0$. There exists $C=C(\lambda,c_0),$ where $c_0$ is the constant in Definition \ref{BG}, such that for any positive solution $w$ of the equation $\Delta w=\lambda w$,     $$
        \varlimsup _{n \rightarrow+\infty}\frac{\ln M_n}{n} =
        \varlimsup _{n \rightarrow+\infty}\max _{B_n\left(x_0\right)} \frac{\ln w}{n} \leq C.
    $$

\end{prop}

\begin{proof}
    By the local Harnack inequality in \cite[Theorem 4.1]{keller2021graphs},
    there exists $C=C(\lambda,c_0) > 1$ such that for any $y \sim x$ with $w(y)>w(x)$,
    \begin{align*}
        C^{-1}\leq \frac{w(y)}{w(x)} \leq C.
    \end{align*}
    This yields that
    $$
        \frac{M_{n+1}}{M_n} \leq C.
    $$ By the iteration, one gets $M_n\leq C^{n-1} M_1.$
  This proves the result.
\end{proof}

\begin{prop}\label{Grow}
    Assume weighed graph $G=(V,E,w,m)$ has bounded geometry,
    $x_0 \in V$, $\lambda > 0.$ There exists $C=C(\lambda, c_0)>1$ such that for any positive non-constant $\lambda-$eigenfunction $w,$
    $$
        \varliminf _{n \rightarrow+\infty}\frac{\ln M_n}{n} =
        \varliminf _{n \rightarrow+\infty} \max _{\partial B_n(x_0)} \frac{\ln w}{n} \geq \ln C.
    $$

\end{prop}

\begin{proof}
   For $n \in \N,$ by the maximum principle, Lemma \ref{Maximal}, we choose $x_n \in \partial B_n\left(x_0\right)$ such that $w\left(x_n\right)= M_n$.
    By the equation of $w$ and the assumption of the bounded geometry,
$$
        \begin{aligned}
            \lambda w\left(x_n\right)\
             & =\sum_{x \sim x_n} \frac{w_{x_n x}}{m_{x_n}}\left(w(x)-w\left(x_n\right)\right) \\
             & \leq c_0^3 \left(\max _{B_{n+1}\left(x_0\right)}w-w\left(x_n\right)\right).
        \end{aligned}
    $$
Thus,
    \begin{align*}
        \frac{M_{n+1}}{M_n} \geq \frac{c_0^3+\lambda}{c_0^3}=:C.
    \end{align*} By the iteration, we prove the result. 
\end{proof}

\section{Proof of the Main result}
In this section, we prove Theorem~\ref{main}.
\begin{proof}
    Without loss of generality, we assume $u > 0$.

    By Proposition \ref{Representation},
    $$
        u(x, t)=\int_{\mathbb{R}} e^{\lambda t} w_\lambda(x) \dd \nu(\lambda),
    $$
    where $w_\lambda$ is a nonnegative solution of the equation $\Delta w=\lambda w$.
    By Remark \ref{Rem} and Lemma~\ref{Spec},
    $$\operatorname{supp}\nu \subset \Lambda = [-\lambda_1(G),+\infty).$$

    Therefore,
    $$
        u(x, t)=\int_{-\lambda_1(G)}^{+ \infty} e^{\lambda t} w_\lambda(x) \dd \nu(\lambda).
    $$

    Next, we argue by contradiction that supp$\ \nu = \left\{0\right\}.$

    Case 1:  $\lambda_1(G)=0$. Then

    $$
        u(x, t)=\int_{0}^{+ \infty} e^{\lambda t} w_\lambda(x) \dd \nu(\lambda).
    $$

    Suppose  $ \exists \
        [a,b] \subset (0,+\infty) \text { s.t. } \nu([a,b]) > 0.$

    By Proposition \ref{Representation}, the function $\lambda \mapsto w_\lambda(x)$ is Borel from $\mathbb{R} \rightarrow \mathbb{R}^{V}$ (the latter space is equipped with the topology of pointwise convergence). 
    Let $\Pi=[a, b]\cap \{\lambda:w_\lambda \neq 0\}$, which is still a Borel set. The restriction of this function on $\Pi$ remains a Borel function. Moreover, by Remark \ref{24}, $\nu(\Pi)=\nu([a, b])>0$.

    Set $
        F:=\{\ln w:  \Delta w=\lambda w, w>0, \lambda \in \Pi\}.
    $
    By Remark \ref{24}, $F$ is not empty.

    Define the function
    $$
        \begin{aligned}
            \varphi:  \Pi & \rightarrow F ,                          \\
            \lambda     & \mapsto \varphi(\lambda)=\ln w_\lambda .
        \end{aligned}
    $$

    By Proposition~\ref{Representation}, this function is well-defined.
    Since the function $w_\lambda \mapsto \ln w_\lambda$ is continuous from $\mathbb{R}^{V}\backslash\{0\} \rightarrow F$ (both with the topology of pointwise convergence),
    $\varphi$ is a Borel function with respect to pointwise convergence topology.

    By Proposition \ref{Grow0}, we can define a new distance function on $F$
    $$\rho_F(f, g):=\max \limits _{n \in \N}\left\{\frac{1}{n} \max \limits _{B_n(x_0)}|f-g|\right\} < +\infty.$$
 
 We claim that $\varphi$ is a Borel function in the new topology induced by the distance function $\rho_F$.
Note that for any $g\in F,$ $c>0,$
    $$\varphi^{-1}\left(f\in F: \left\{\rho_F(f, g)
        \leq c\right\}\right)=\bigcap_{n \in \N}\left\{\lambda \in \mathbb{R}: \max _{B_n(x_0)}|\varphi(\lambda)-g|
        \leq cn \right\}.$$

    One easily checks that the function $\varphi(\lambda) \mapsto
        \max \limits _{B_n\left(x_0\right)}|\varphi(\lambda)-g|$ is a continuous function from $F \rightarrow \mathbb{R}$ (where the topology of $F$ is induced by pointwise convergence).
    Thus the function $\lambda \mapsto \max \limits _{B_n\left(x_0\right)}|\varphi(\lambda)-g|$ is a Borel function, and hence $\varphi^{-1}\left(f \in F: \left\{\rho_F(f, g)
        \leq c\right\}\right)$ is Borel. This proves the claim.

    By Rusin's theorem, there exists a compact set $K \subset \Pi$ such that
    $
        \nu(K)>0, \varphi |_K
    $
    is a continuous function.
    Thus, there exists $\lambda_0 \in K$ such that $\nu\left(B_r\left(\lambda_0\right) \cap K\right)>0,  \forall r>0 $.
    Since $\varphi |_K$ is continuous at $\lambda_0$, $ \forall \ \varepsilon>0 $, $ \exists \ \delta>0 $ such that for any $ \lambda \in B_\delta(\lambda_0)\cap K$, 
    $$
   \frac{1}{n} \max _{ B_n\left(x_0\right)}\left|\varphi\left(\lambda_0\right)-\varphi(\lambda)\right| =\frac{1}{n} \max _{x\in B_n\left(x_0\right)}\left|\ln w_{\lambda_0}(x)-\ln w_\lambda(x)\right|          \leq \varepsilon, \quad \forall n \in \N.
    $$

    Hence for any $ \lambda \in B_\delta(\lambda_0)\cap K$, $n \in \N$ and $x \in  B_n\left(x_0\right)$, 
    $$
        \ln w_{\lambda_0}(x)\leq \ln w_\lambda(x)+n \varepsilon.
 $$

    Taking the average of $\lambda$ over $B_\delta\left(\lambda_0\right) \cap K$, and by Jensen's inequality,

    $$
        \begin{aligned}
        \ln w_{\lambda_0}(x) &\leq
        \frac{1}{\nu\left(B_\delta\left(\lambda_0\right) \cap K\right)}
        \int_{B_\delta {\left(\lambda_0\right) \cap K}} \ln w_\lambda(x) \dd \nu ( \lambda)+n \varepsilon,\\
             & \leq \ln \left( \frac{1}{\nu\left(B_\delta\left(\lambda_0\right) \cap K \right)}
            \int_{B_\delta {\left(\lambda_0\right) \cap K}} w_\lambda(x) \dd \nu ( \lambda) \right) +n \varepsilon, \\
             & \leq \ln \left( \frac{1}{\nu\left(B_\delta\left(\lambda_0\right) \cap K \right)}
            \int_{\mathbb{R}} w_\lambda(x) \dd \nu(\lambda)\right)+n \varepsilon .
        \end{aligned}
    $$

    For fixed $ t \in (-\infty, t_0)$, by Proposition \ref{Representation}, we have
    $$
        \int_{\mathbb{R}} w_\lambda(x) d \nu(\lambda)=e^{-\lambda t} \int_{\mathbb{R}} e^{\lambda t}
        w_\lambda(x) \dd \nu(\lambda)=e^{-\lambda t} u(x, t).
    $$

    Hence, for all $x \in B_n\left(x_0\right)$, 
    $$
        \ln w_{\lambda_0}(x) \leq -\ln \nu \left(B_\delta\left(\lambda_0\right) \cap K\right)-\lambda t+\ln u(x, t) + n \varepsilon.
    $$

    Dividing on both sides by $n$, we get
    $$
        \frac{\ln w_{\lambda_0}(x)}{n} \leq \frac{\ln u\left(x, t\right)}{n}+\varepsilon
        +\frac{c}{n},
    $$
    where $c$ is a constant. Letting $x=x_n\in \partial B_n(x_0),$ which attains its maximum of $w_{\lambda_0}$ on $B_n(x_0),$ we have
    $$
        \frac{1}{n} \max_{B_n(x_0)} \ln w_{\lambda_0} \leq \frac{\ln u\left(x_n, t\right)}{n}+\varepsilon +\frac{c}{n} =\frac{\ln u\left(x_n, t\right)}{\rho(x_n,x_0)}+\varepsilon +\frac{c}{n} .
    $$

    Letting $n \to \infty$, by Lemma \ref{Grow}, and considering the assumption of spatial subexponential growth, we obtain
    $$
        0<\ln C \leq \varlimsup _{n \rightarrow +\infty} \frac{1}{n} \max _{B_n(x_0)} \ln w_{\lambda_0} \leq \varlimsup _{n \rightarrow +\infty} \frac{\ln u\left(x_n,t\right)}{\rho\left(x_n, t\right)}+\varepsilon = \varepsilon.
    $$

    Since $\varepsilon $ is arbitrary, this yields a contradiction.

  Hence, supp$\ \nu = \left\{0\right\}$, and consequently
    $$
        \begin{aligned}
            u(x, t) & =\int_{\mathbb{R}} e^{\lambda t} w_\lambda(x) \dd \nu(\lambda) \\
                    & =e^{0 \cdot t} w_0(x)                                          \\
                    & =w_0(x) .
        \end{aligned}
    $$

 That is, $u$ is independent of $t$, and is a harmonic function on $G$ with respect to $x$.

    \vspace{0.5em}

    Case 2: $\lambda_1(G) > 0$. Then
    $$
        u(x, t)=\int_{-\lambda_1}^{+ \infty} e^{\lambda t} w_\lambda(x) \dd \nu(\lambda).
    $$

    First, we argue by contradiction that $\operatorname{supp} \nu \subset [0,+\infty)$.
    Suppose it is not true, then there exists $[a, b] \subset (-\infty, 0)$ such that $\nu([a, b]) > 0$.
    Thus,
    $$
        \begin{aligned}
            u(x, t) & \geq \int_a^b e^{\lambda t} w_\lambda(x) \dd \nu(\lambda) \\
                    & \geq e^{a t} \int_a^b w_\lambda(x) \dd \nu(\lambda).
        \end{aligned}
    $$

    By using Rusin's theorem, we can prove $\int_a^b w_\lambda(x) \dd \nu(\lambda) > 0$.
    However, for fixed $x$, it contradicts with the assumption of subexponential growth with respect to time.
    Hence, we conclude that
    $$
        \operatorname{supp} \nu \subset [0,+\infty).
    $$

    The remaining part of the argument follows verbatim as in Case 1.

\end{proof}

\bigskip \textbf{Acknowledgements.} 
We thank Professor Qi S. Zhang for his suggestion of the problem and many helpful discussions.
B. Hua is supported by NSFC, no.11831004, and by Shanghai Science and Technology Program [Project No. 22JC1400100].

\bibliography{Liouville_Ancient_Heat}

\newcommand{\etalchar}[1]{$^{#1}$}
\begin{thebibliography}{BHL{\etalchar{+}}15}

\bibitem[And83]{Anderson83}
M.~Anderson.
\newblock {The Dirichlet problem at infinity for manifolds of negative curvature}.
\newblock {\em J. Differential Geom.}, 18(4):701--721, 1983.

\bibitem[Ben91]{Benjamini91}
I.~Benjamini.
\newblock {Instability of the Liouville property for quasi-isometric graphs and manifolds of polynomial volume growth}.
\newblock {\em J. Theoret. Probab.}, 4(3):631--637, 1991.

\bibitem[BHL{\etalchar{+}}15]{bauer2015li}
Frank Bauer, Paul Horn, Yong Lin, Gabor Lippner, Dan Mangoubi, and Shing-Tung Yau.
\newblock Li-{Y}au inequality on graphs.
\newblock {\em J. Differential Geom.}, 99(3):359--405, 2015.

\bibitem[Bri13]{Brighton13}
K.~Brighton.
\newblock {A Liouville-type theorem for smooth metric measure spaces}.
\newblock {\em J. Geom. Anal.}, 23(2):562--570, 2013.

\bibitem[Cho69]{choquet1969lectures}
Gustave Choquet.
\newblock {\em Lectures on analysis. {V}ol. {I}: {I}ntegration and topological vector spaces}.
\newblock W. A. Benjamin, Inc., New York-Amsterdam, 1969.
\newblock Edited by J. Marsden, T. Lance and S. Gelbart.

\bibitem[CM97]{ColdingMAnnals97}
Tobias~H. Colding and William~P. Minicozzi, II.
\newblock Harmonic functions on manifolds.
\newblock {\em Ann. of Math. (2)}, 146(3):725--747, 1997.

\bibitem[CM21]{colding2021optimal}
Tobias~Holck Colding and William~P. Minicozzi, II.
\newblock Optimal bounds for ancient caloric functions.
\newblock {\em Duke Math. J.}, 170(18):4171--4182, 2021.

\bibitem[CY75]{ChengYau75}
S.~Y. Cheng and S.~T. Yau.
\newblock {Differential equations on Riemannian manifolds and their geometric applications}.
\newblock {\em Comm. Pure Appl. Math.}, 28(3):333--354, 1975.

\bibitem[Del98]{DelmottePolynomial98}
Thierry Delmotte.
\newblock Harnack inequalities on graphs.
\newblock In {\em S\'{e}minaire de {T}h\'{e}orie {S}pectrale et {G}\'{e}om\'{e}trie, {V}ol. 16, {A}nn\'{e}e 1997--1998}, volume~16 of {\em S\'{e}min. Th\'{e}or. Spectr. G\'{e}om.}, pages 217--228. Univ. Grenoble I, Saint-Martin-d'H\`eres, [1998].

\bibitem[Ers04]{Erschler04}
A.~Erschler.
\newblock {Liouville property for groups and manifolds}.
\newblock {\em Invent. Math.}, 155(1):55--80, 2004.

\bibitem[Gri90]{Grigoryan90}
A.~Grigor'yan.
\newblock {Dimension of spaces of harmonic functions}.
\newblock {\em Math. Notes}, 48:1114--1118, 1990.

\bibitem[Gri91]{Grigoryan91}
A.~Grigor'yan.
\newblock {The heat equation on noncompact Riemannian manifolds (Russian)}.
\newblock {\em Mat. Sb.}, 182(1):55--87 (English translation in Math. USSR--Sb. 72(1): 47--77, 1992.), 1991.

\bibitem[Gri18]{grigor2018introduction}
Alexander Grigor'yan.
\newblock {\em Introduction to analysis on graphs}, volume~71 of {\em University Lecture Series}.
\newblock American Mathematical Society, Providence, RI, 2018.

\bibitem[HLLY19]{HLLY19}
Paul Horn, Yong Lin, Shuang Liu, and Shing-Tung Yau.
\newblock Volume doubling, {P}oincar\'{e} inequality and {G}aussian heat kernel estimate for non-negatively curved graphs.
\newblock {\em J. Reine Angew. Math.}, 757:89--130, 2019.

\bibitem[Hua19]{hua2019Liouville}
B.~Hua.
\newblock Liouville theorem for bounded harmonic functions on manifolds and graphs satisfying non-negative curvature dimension condition.
\newblock {\em Calc. Var. Partial Differential Equations}, 58(2):Art. 42, 8, 2019.

\bibitem[JMR19]{jost2019liouville}
J{\"u}rgen Jost, Florentin M{\"u}nch, and Christian Rose.
\newblock Liouville property and non-negative ollivier curvature on graphs.
\newblock {\em arXiv preprint arXiv:1903.10796}, 2019.

\bibitem[KLW21]{keller2021graphs}
Matthias Keller, Daniel Lenz, and Rados\l aw~K. Wojciechowski.
\newblock {\em Graphs and discrete {D}irichlet spaces}, volume 358 of {\em Grundlehren der mathematischen Wissenschaften [Fundamental Principles of Mathematical Sciences]}.
\newblock Springer, Cham, [2021] \copyright 2021.

\bibitem[Li12]{Li12}
Peter Li.
\newblock {\em Geometric analysis}, volume 134 of {\em Cambridge Studies in Advanced Mathematics}.
\newblock Cambridge University Press, Cambridge, 2012.

\bibitem[LLY17]{lin2017gradient}
Yong Lin, Shuang Liu, and Yunyan Yang.
\newblock A gradient estimate for positive functions on graphs.
\newblock {\em J. Geom. Anal.}, 27(2):1667--1679, 2017.

\bibitem[LT87]{LiTam87}
P.~Li and L.-F. Tam.
\newblock {Positive harmonic functions on complete manifolds with nonnegative curvature outside a compact set}.
\newblock {\em Ann. of Math. (2)}, 125(1):171--207, 1987.

\bibitem[LY86]{LiYau86}
P.~Li and S.~T. Yau.
\newblock {On the parabolic kernel of the Schroedinger operator}.
\newblock {\em Acta Math.}, 156(3--4):153--201, 1986.

\bibitem[Lyo87]{Lyons87}
T.~Lyons.
\newblock {Instability of the Liouville property for quasi-isometric Riemannian manifolds and reversible Markov chains}.
\newblock {\em J. Differential Geom.}, 26(1):33--66, 1987.

\bibitem[LZ19]{lin2019ancient}
Fanghua Lin and Q.~S. Zhang.
\newblock On ancient solutions of the heat equation.
\newblock {\em Comm. Pure Appl. Math.}, 72(9):2006--2028, 2019.

\bibitem[Mos21]{mosconi2021}
Sunra Mosconi.
\newblock Liouville theorems for ancient caloric functions via optimal growth conditions.
\newblock {\em Proc. Amer. Math. Soc.}, 149(2):897--906, 2021.

\bibitem[Rud91]{rudin1973functional}
Walter Rudin.
\newblock {\em Functional analysis}.
\newblock International Series in Pure and Applied Mathematics. McGraw-Hill, Inc., New York, second edition, 1991.

\bibitem[SC92]{Saloff92}
L.~Saloff-Coste.
\newblock {A note on Poincar\'e, Sobolev, and Harnack inequalities}.
\newblock {\em Internat. Math. Res. Notices}, (2):27--38, 1992.

\bibitem[Sul83]{Sullivan83}
D.~Sullivan.
\newblock {The Dirichlet problem at infinity for a negatively curved manifold}.
\newblock {\em J. Differential Geom.}, 18(4):723--732, 1983.

\bibitem[SZ06]{souplet2006sharp}
Philippe Souplet and Qi~S. Zhang.
\newblock Sharp gradient estimate and {Y}au's {L}iouville theorem for the heat equation on noncompact manifolds.
\newblock {\em Bull. London Math. Soc.}, 38(6):1045--1053, 2006.

\bibitem[Wan02]{WangFY02}
F.~Y. Wang.
\newblock {Liouville theorem and coupling on negatively curved manifolds}.
\newblock {\em Stochastic Process. Appl.}, 100:27--39, 2002.

\bibitem[Yau75]{Yau75}
S.~T. Yau.
\newblock {Harmonic functions on complete Riemannian manifolds}.
\newblock {\em Comm. Pure Appl. Math.}, 28:201--228, 1975.

\end{thebibliography}
\bibliographystyle{alpha}

\end{document}